\documentclass[10pt,reqno]{amsart}
\allowdisplaybreaks
\usepackage{amsmath,amstext,amssymb,epsfig}
\usepackage{graphicx}
\usepackage{mathrsfs}
\calclayout
\makeatletter
\renewcommand{\@seccntformat}[1]{\bf\csname the#1\endcsname.}
\renewcommand{\section}{\@startsection{section}{1}
	\z@{.7\linespacing\@plus\linespacing}{.5\linespacing}
	{\normalfont\upshape\bfseries\centering}}
\renewcommand{\@biblabel}[1]{\@ifnotempty{#1}{#1.}}
\makeatother
\theoremstyle{plain}
\newtheorem{thm}{Theorem}[section]
\newtheorem{lem}[thm]{Lemma}
\newtheorem{prop}[thm]{Proposition}

\theoremstyle{definition}

\newtheorem{defn}[thm]{Definition}
\newtheorem{rem}{Remark}[section]

\usepackage[parfill]{parskip}
\usepackage[german]{varioref}
\usepackage[all]{xy}
\usepackage{cancel}

\def\>{\succ}
\def\<{\prec}

\begin{document}
	\title[AFI Maha\textsuperscript{1}, Sania Asif\textsuperscript{2}, Chouaibi Sami\textsuperscript{3}, Basdouri Imed\textsuperscript{4}]{On the cohomology based on the generalized representations of $n$-Lie Algebras}
	\author{AFI Maha\textsuperscript{1}, Sania Asif\textsuperscript{2}, Chouaibi Sami\textsuperscript{3}, Basdouri Imed\textsuperscript{4}}
	\address{\textsuperscript{1}University of Sfax, Faculty of sciences of Sfax, departement of Mathematics, BP 1171, Sfax, Tunisia.} \address{\textsuperscript{2}School of Mathematics and statistics, Nanjing University of Information science and technology, Nanjing, Jiangsu Province, PR China.} \address{\textsuperscript{3}University of Sfax, Faculty of sciences  of Sfax, departement of Mathematics, BP 1171, Sfax, Tunisia.} \address{\textsuperscript{4} University of Gafsa, Faculty of sciences  of Gafsa, departement of Mathematics, Zarroug 2112, Gafsa, Tunisia.}
	\email{maha\_2015afi@hotmail.com\textsuperscript{1}}
	\email{11835037@zju.edu.cn\textsuperscript{2}}\email{chouaibi\_sami@yahoo.fr\textsuperscript{3}}\email{basdourimed@yahoo.fr\textsuperscript{4}}
	\date{\today}
	\thanks{
	}
	\begin{abstract}In the present paper, we define the new class of representation on $n$-Lie algebra that is called as generalized representation. We study the cohomology theory corresponding to generalized representations of $n$-Lie algebras and show its relation with the cohomology corresponding to the usual representations. Furthermore, we provide the computation for the low dimensional cocycles.
	\end{abstract}
	\footnote{
	}
	\maketitle

 \section{Introduction}
The notion to the $ n $-ary operation, or more precisely the ternary operation, is well known since 19th century.  It was first introduced by A. Cayley in its study about cubic matrix. By definition, the $ n $-ary algebra is a $ \mathbb{K} $-vector space equipped with an $ n $-linear map. The most well known $ n $-ary algebra is the $ n $-Lie algebra that is determined by Filippov in 1985. In \cite{F}, Filippov studied that an $ n $-Lie algebra is a generalization of the Lie algebra (with $n=2$), where the regular Lie bracket is replaced by the $ n $-ary bracket. Filippov also studied the solvability and nilpotency of the $ n $-Lie algebras. Moreover, $n$-Lie algebra is an algebraic structure related to Nambu mechanics, where the Hamiltonian system is based on a ternary product see \cite{NY}. For the algebraic formulation of Nambu mechanics see \cite{G, T}.
 \\\hspace*{0.5cm}
There has been a growing interest in the study of $n$-Lie algebra. The series of numerous researches over the years have  contributed a lot to recognize the structural theory of $n$-Lie algebra see \cite{BGS,CT,H,ST}. More specifically, Bai in \cite{BW}, studied the construction of $3$-Lie algebra by using linearity of the ordinary Lie algebra. The classification of the $(n+2)$-dimensional $n$-Lie algebras and the representation theory of $n$-Lie algebras were studied in \cite{BSZ}, where the adjoint representation was introduced with the aid of $n$-bracket having $(n-1)$ fixed elements. The  deformation theory of $n$-Lie algebras was studied in \cite{FJ, MA, TL}. In addition, the structure of Leibniz algebras was obtained by considering the notion of fundamental objects in \cite{DT}. The structure of a graded Lie algebra on the cochain complex of a $n$-Leibniz algebra was further determined in \cite{RM}, where the author defined the  $n$-Leibniz algebra as a canonical structure, see \cite{AI2} for more details.

 \par In the present paper, we study the new class of representation that is called a generalized representation of the $n$-Lie algebra $\mathfrak{g}$ over the vector space $V$. We also study cohomology of $n$-Lie algebra with coefficients in the generalized representation. Furthermore, we compute its $1$-cocycles, and $2$-cocycles. The paper is organized as follows.
In section $2$, we recall some important definitions and notions of $n$-Lie algebra. By using the usual representation we define the generalized representations of $n$-Lie algebra. This leads us to discuss a new class of Cohomology of $n$-Lie algebras. Furthermore, we describe $n$-Lie algebra's structure by using graded Lie algebra structure. In section $3$, we establish the cohomology related to the new representation of $n$-Lie algebra that is called generalized representation. Later on we determine the computation of $1$-cocycles, and $2$-cocycles of $n$-Lie algebra.
\par In this paper, all the
vector spaces are over $\mathbb{K}$, where $\mathbb{K}$ is algebraically closed field  with the characteristic $0$.
 \section{$n$-Lie algebras and their representations}
In this section we first recall some important definitions  and notions about $n$-Lie algebra. Later, we define generalized representation of $n$-Lie algebra and discuss its cohomology space. At the end of this section, we study the graded Lie algebra associated to $n$-Lie algebra.
\begin{defn}\label{m1}\cite{LSZB}
~~\\
\rm
A $n$-Lie algebra $\mathfrak{g}$ is a vector space together with an $n$-multilinear skew-symmetric bracket $[\cdot,~\cdots,~\cdot]_{\mathfrak{g}}: \wedge^{n}\mathfrak{g}\rightarrow \mathfrak{g}$ such that for all $x_{i},~y_{i} \in \mathfrak{g}$, the following Filippov identity holds:
\begin{equation}\label{1}
 [x_{1},~\cdots,~x_{n-1},~[y_{1},~\cdots,~y_{n}]_{\mathfrak{g}}]_{\mathfrak{g}}=\displaystyle\sum^{n}_{i=1}
 [y_{1},~\cdots,~y_{i-1},~[x_{1},~\cdots,~x_{n-1},~y_{i}]_{\mathfrak{g}},~y_{i+1},~\cdots,~y_{n}]_{\mathfrak{g}}.
\end{equation}This identity is also known as $\mathbf{Fundamental~ Identity}~(\mathbf{FI}).$
\end{defn}
\begin{defn}\label{m2}\cite{AI2}
~~\\
\rm
A skew-symmetric element $X \in \wedge^{n-1}\mathfrak{g}$ is called fundamental object of a $n$-Lie algebra or $\mathbf{FA}$~(Filippov algebra). It is determined by $(n-1)$ elements $x_{1},~\cdots,~x_{n-1}$ of $\mathfrak{g}$.
Let $X,~Y \in \wedge^{n-1}\mathfrak{g}$ be two fundamental objects of $n$-Lie algebra, then (non associative composition ) $[X,~Y]_{F}\in  \wedge^{n-1}\mathfrak{g}$  of these two objects  is the bi and $i$-linear map $\wedge^{n-1}\mathfrak{g}\otimes \wedge^{n-1}\mathfrak{g}\rightarrow \wedge^{n-1}\mathfrak{g}$, given by the sum
\begin{equation}\label{2}
  [X,~Y]_{F} = \displaystyle\sum^{n-1}_{i=1}(y_{1},~\cdots,~X\cdot y_{i},~\cdots,~y_{n-1})=\displaystyle\sum^{n-1}_{i=1} (y_{1},~\cdots,~[x_{1},~\cdots,~x_{n-1},~y_{i}]_{\mathfrak{g}},~\cdots,~y_{n-1}),
\end{equation}
with $X= x_{1}\wedge\cdots\wedge x_{n-1} $ and $ Y = y_{1}\wedge\cdots\wedge y_{n-1}.$
\end{defn}

\begin{defn}\label{m4}\cite{BSZ}
   ~~\\
\rm
 A representation $\rho$ of a $n$-Lie algebra $\mathfrak{g}$ on a vector space $V$ is a multilinear map $\rho: \wedge^{n-1}\mathfrak{g} \rightarrow \mathfrak{gl}(V)$, such that for all $X,~Y\in \wedge^{n-1}\mathfrak{g},~ x_{i},~y_{i} \in \mathfrak{g}$, the following equalities hold:
\begin{equation}\label{3,4}
 \rho(X)\rho(Y)-\rho(Y)\rho(X)= \displaystyle\sum^{n-1}_{i=1} \rho(y_{1},~\cdots,~y_{i-1},~[x_{1},~\cdots,~x_{n-1},~y_{i}]_{\mathfrak{g}},~y_{i+1},~\cdots,~y_{n-1}),~
\end{equation} \begin{equation}
   \rho(x_{1},~\cdots,~x_{n-2},~[y_{1},~\cdots,~y_{n}]_{\mathfrak{g}})=\displaystyle\sum^{n}_{i=1}(-1)^{n-i}\rho(y_{1},~\cdots,~\hat{y_{i}},~\cdots,~y_{n})
\rho(x_{1},~\cdots,~x_{n-2},~y_{i}),
\end{equation}
where $\hat{y_{i}}$ shows the omission of  $y_{i}$.
\end{defn}

\begin{defn}\label{m12}
  ~~\\
\rm
A generalized representation of a $n$-Lie algebra $\mathfrak{g}$ on a vector space V consists of linear maps $\rho:\wedge^{n-1}\mathfrak{g}\rightarrow \mathfrak{gl}(V)$ and $\vartheta:\mathfrak{g}\rightarrow Hom(\wedge^{n-1}V,~V)$, such that
$$[\pi+\bar{\rho}+\bar{\vartheta},~ \pi+\bar{\rho}+ \bar{\vartheta}]^{nLie}=0,$$
where $\bar{\vartheta}:\wedge^{n}(\mathfrak{g}\oplus V) \rightarrow \mathfrak{g}\oplus V$ is induced by $\vartheta$ via
\begin{equation}\label{10}
\bar{\vartheta} (x_{1}+v_{1},~\cdots,~x_{n}+v_{n} )=\displaystyle\sum_{i=1}^{n}(-1)^{n-i}\vartheta(x_{i})(v_{1}\wedge\cdots\wedge\hat{v_{i}}\wedge\cdots\wedge v_{n}),\end{equation}
$\bar{\rho}: \wedge^{n}(\mathfrak{g}\oplus V) \rightarrow \mathfrak{g}\oplus V$ is induced by $\rho$ via
 \begin{equation}\label{10.}
\bar{\rho} (x_{1}+v_{1},~\cdots,~x_{n}+v_{n} )=\displaystyle\sum_{i=1}^{n}(-1)^{n-i}\rho(x_{1},~\cdots,~\hat{x_{i}},~\cdots,~x_{n})(v_{i}),
\end{equation}
$$\forall ~x_{i}\in \mathfrak{g}, ~v_{i}\in V ~and~ i\in \{1,~\cdots,~n\}.$$
And  the map $\pi: \wedge^{n}\mathfrak{g}\rightarrow \mathfrak{g}$ defines $n$-Lie bracket.\\
We denote the  generalized representation by a triplet $(V; ~\rho, ~\vartheta)$.
\end{defn}A $p$-cochain on $\mathfrak{g}$ with the c{\oe}fficients in a representation $(V;~\rho)$ is a linear map
$$ \alpha^{p}: \wedge^{n-1}\mathfrak{g}\overbrace{\otimes \cdots \otimes }^{(p-1)}\wedge^{n-1}\mathfrak{g}\wedge \mathfrak{g} \rightarrow V.$$
We denote the space of $p$-cochains by $C^{p-1}(\mathfrak{g};V)$. The coboundary operator $\delta_{\rho}: C^{p-1}(\mathfrak{g};V)\rightarrow C^{p}(\mathfrak{g};V)$ is given by

 $\delta_{\rho} \alpha^{p}(\chi_{1},~\cdots,~\chi_{p},~z) $
\begin{eqnarray}\label{6}
&=&\displaystyle\sum^{}_{1\leq i<k}(-1)^{i}\alpha^{p}(\chi_{1},~\cdots,~\hat{\chi_{i}} ,~\cdots,~\chi_{k-1},~[\chi_{i},~\chi_{k}]_{F},~\chi_{k+1},~\cdots,~\chi_{p},~z)\nonumber\\
 &&+ \displaystyle\sum^{p}_{i=1}(-1)^{i}\alpha^{p}(\chi_{1},~\cdots,~\hat{\chi_{i}},~\cdots,~\chi_{p},~[\chi_{i},~z])\nonumber\\
&&+\displaystyle\sum^{p}_{i=1}(-1)^{i+1}\rho(\chi_{i})\alpha^{p}(\chi_{1},~\cdots,~\hat{\chi_{i}},~\cdots,~\chi_{p},~z)\nonumber\\
&&+\displaystyle\sum^{n-1}_{i=1}(-1)^{n+p-i+1}\rho(x^{1}_{p},~x^{2}_{p},~\cdots,~ \hat{x^{i}_{p}},~\cdots,~x^{n-1}_{p},~z)\alpha^{p}(\chi_{1},~\cdots,~\chi_{p-1},~x^{i}_{p}),
\end{eqnarray}
for all $\chi_{i}=(x^{1}_{i},~  x^{2}_{i},~\cdots,~ x^{n-1}_{i})\in \wedge^{n-1}\mathfrak{g}$ and $z\in \mathfrak{g}$.\\
 If the set of $p$-cocycles and the set of $p$-coboundaries are denoted by $Z^{p}(\mathfrak{g};~ V)$ and $B^{p}(\mathfrak{g};~ V)$ respectively, then the $p$-th cohomology space is $$H^{p}( \mathfrak{g};~ V)= Z^{p}(\mathfrak{g};~ V)/ B^{p}(\mathfrak{g};~ V).$$
For more detail about the cohomology of $n$-Lie algebra see \cite{AI1}.\\
The structure of a graded Lie algebra was determined in \cite{RM}, where the  $n$-Leibniz structure was further described as a canonical structure. The accurate formulas of $n$-Lie algebra are as follows;

Let $L_{p}=C^{p}(\mathfrak{g},~\mathfrak{g})= Hom(\wedge^{n-1}\mathfrak{g}\overbrace{\otimes\cdots \otimes}^{(p)}\wedge^{n-1}\mathfrak{g}\wedge\mathfrak{g},~\mathfrak{g})$ and $L= \oplus_{p\geq0} L_{p}$. Let $\alpha \in C^{p}(\mathfrak{g},~\mathfrak{g}),~ \beta \in C^{q}(\mathfrak{g},~\mathfrak{g}),~p,~q\geq0$. Let $\mathfrak{X}_{i}=x_{i}^{1}\wedge \cdots \wedge x_{i}^{n-1}\in  \wedge^{n-1}\mathfrak{g}$ for $i=1,~\cdots,~ p+q$ and $z\in \mathfrak{g}$.
For each subset $J=\{j_{1},~\cdots,~j_{q+1}\}_{j_{1}<\cdots<j_{q+1}}\subset N\triangleq \{1,~2,~\cdots,~p+q+1\},$\\ let $I=\{i_{1},~\cdots,~i_{p}\}_{i_{1}<\cdots<i_{p}}=N / J$, then we have

\begin{thm}\label{m7}\cite{RM}
  ~~\\
\rm
The graded vector space $L$ equipped with the graded commutator bracket
\begin{align}\label{7}
  [\alpha,~\beta]^{nLie}=(-1)^{pq}\alpha\circ\beta -\beta\circ\alpha,
\end{align}
is a graded Lie algebra and $\alpha\circ\beta\in L^{p+q}$ is defined by\\
$\alpha\circ\beta (\mathfrak{X}_{1},~\cdots,~ \mathfrak{X}_{p+q},~z)=\displaystyle\sum_{J,~j_{q+1}<p+q+1}(-1)^{(J,~I)}\sum_{s=1} ^{n-1} (-1)^{k}\\
 \alpha( \mathfrak{X}_{i_{1}},~\cdots,~ \mathfrak{X}_{i_{k}},~x_{j_{q+1}}^{1}~ \wedge~\cdots~\wedge ~x_{j_{q+1}}^{s-1}~\wedge~ \beta(\mathfrak{X}_{j_{1}},~\cdots,~ \mathfrak{X}_{j_{q}},~x_{j_{q+1}}^{s} )~\wedge~ x_{j_{q+1}}^{s+1} ~ \wedge ~\cdots~\wedge ~x_{j_{q+1}}^{n-1},~ \mathfrak{X}_{i_{k+1}},~\cdots,~  \mathfrak{X}_{i_{p-1}},~z)\\ +\displaystyle\sum_{J,~j_{q+1}=p+q+1}(-1)^{(J,~I)}(-1)^{p}\alpha(\mathfrak{X}_{i_{1}},~\cdots,~ \mathfrak{X}_{i_{p}},~\beta(\mathfrak{X}_{j_{1}},~\cdots,~ \mathfrak{X}_{j_{q}},~z)),$\\
where $k$ is uniquely determined by the condition $i_{k}<j_{q+1}<i_{k+1}$. \begin{itemize}
	\item If $j_{q+1} <i_{1}$, i.e. $j_{q+1}=q+1,~i_{1}=q+2$ then $k=0$.
	\item If $j_{q+1} >i_{p-1},$ i.e. $j_{q+1}=p+q $ then $k=p-1.$
\end{itemize}
We can use the graded Lie algebra structure $(L,~[\cdot,~\cdot]^{nLie})$ to describe $n$-Lie algebra structures as well as coboundary operators.
\end{thm}
\begin{lem}\label{m8}
 ~~\\
\rm
  The map $\pi: \wedge^{n}\mathfrak{g}\rightarrow \mathfrak{g}$ defines $n$-Lie bracket if and only if $[\pi,~\pi]^{nLie}=0$, i.e. $\pi$ is canonical structure.
\end{lem}
\begin{rem}\label{m9}
 ~~\\
\rm
  Let $\mathfrak{g}$ be a $n$-Lie algebra. For $x_{1},~\cdots,~x_{n-1}\in \mathfrak{g}$, define an adjoint map $ad:\wedge^{n-1}\mathfrak{g}\rightarrow \mathfrak{gl}(\mathfrak{g})$ by
$$ ad_{x_{1},~\cdots,~x_{n-1}} y=[x_{1},~\cdots,~x_{n-1},~y]_{\mathfrak{g}},~~~~\forall y\in \mathfrak{g}.$$
The map $ad$ defines an adjoint representation of $\mathfrak{g}$.
The coboundary operator related to this representation is denoted by $\delta_{\mathfrak{g}}$.
\end{rem}

\begin{lem}\label{m10}
  ~~\\
\rm
If $\pi:\wedge^{n}\mathfrak{g}\rightarrow \mathfrak{g}$ is a $n$-Lie bracket, then we have
\begin{align}\label{8}
 [\pi,~\alpha]^{nLie } &=\delta_{\mathfrak{g}}(\alpha),~~~~\forall \alpha\in C^{p}(\mathfrak{g},~\mathfrak{g}),~ p\geq0.
\end{align}

\end{lem}

 \section{New Cohomology Complex of $n$-Lie Algebras}
The purpose of this section is to construct new class of cohomology of $n$-Lie algebras by using the generalized representations. \\Let  $(\mathfrak{g},~[\cdot,~\cdots,~\cdot]_{\mathfrak{g}})$ be $n$-Lie algebras and $(V;~\rho,~\vartheta)$ be a generalized representation of $\mathfrak{g}$. We set
$C^{p} _{>}(\mathfrak{g}\oplus V,~V)$ to be the set of $(p+1)$-cochains, which is defined as a subset of $C^{p} (\mathfrak{g}\oplus V,~V) $
such that
\begin{equation}\label{}
 C^{p} (\mathfrak{g}\oplus V,~V)=  C^{p} _{>}(\mathfrak{g}\oplus V,~V)\oplus C^{p} ( V,~V).
\end{equation}
With   $ C^{p} (\mathfrak{g}\oplus V,~V)=\{\alpha: \wedge^{n-1}(\mathfrak{g}\oplus V) \overbrace{\otimes  \cdots \otimes} ^{p~ times} \wedge^{n-1}(\mathfrak{g}\oplus V) \wedge (\mathfrak{g}\oplus V)\rightarrow V \}.$ \\
By direct calculation, we have
$$[\pi+\bar{\rho}+\bar{\vartheta},~C^{p}_{>}(\mathfrak{g}\oplus V,~V)]^{nLie}\subseteq  C^{p+1}_{>}(\mathfrak{g}\oplus V,~V)~~~ \forall~ p\geq 0.$$
Indeed, let\\
 $$\left\{\begin{array}{ll}\alpha \in C^{p}_{>}(\mathfrak{g}\oplus V,~V),~ \textit{then} ~\alpha \in C^{p}(\mathfrak{g}\oplus V,~V)~ \textit{such that }~\alpha|_{C^{p} ( V,~V)}=0& \hbox{;} \\
        \pi+\bar{\rho}+\bar{\vartheta}\in C^{1}(\mathfrak{g}\oplus V,~V), ~such ~that~  \pi+\bar{\rho}+\bar{\vartheta}|_{C^{1} ( V,~V)}=0 & \hbox{.}
      \end{array}
    \right.$$\\
Hence by applying Theorem  \ref{m7} we have \\
$[ \pi+\bar{\rho}+\bar{\vartheta},~\alpha]^{nLie} \in C^{p+1}(\mathfrak{g}\oplus V,~V)$  such that $[ \pi+\bar{\rho}+\bar{\vartheta},~\alpha]^{nLie}|_{C^{p+1} ( V,~V)}=0, $ \\
then $[ \pi+\bar{\rho}+\bar{\vartheta},~\alpha]^{nLie}\in  C^{p+1}_{>}(\mathfrak{g}\oplus V,~V)   $

Define another operator $$d: C^{p}_{>}(\mathfrak{g}\oplus V,~V)\rightarrow C^{p+1}_{>}(\mathfrak{g}\oplus V,~V)$$ such that
\begin{align}
d(\alpha)&:=[ \pi+\bar{\rho}+\bar{\vartheta},~\alpha]^{nLie}  ~~~ \alpha\in C^{p}_{>}(\mathfrak{g}\oplus V,~V).
\end{align}
\begin{prop}\label{m20}
   ~~\\
\rm
Let $\mathfrak{g}$ be a $n$-Lie algebra, and $(V;~\rho,~\vartheta)$ be a generalized representation of $\mathfrak{g}$. Then $d\circ d=0.$ Therefore, we obtain a new cohomology complex, where $ C^{p-1}_{>}(\mathfrak{g}\oplus V,~V)$ is the space of $p$-cochains.
\end{prop}

\begin{proof}
   ~~\\
\rm In order to show  $d\circ d(\alpha)=0,$
consider that $\alpha \in C^{p-1}_{>}(\mathfrak{g}\oplus V,~V).$
By calculating the graded Jacobi identity associated to the  graded commutator bracket in Theorem \ref{m7}, we have
\begin{align*}
  d\circ d (\alpha) &=[\pi+\bar{\rho}+\bar{\vartheta},~d(\alpha)]^{nLie}=[\pi+\bar{\rho}+\bar{\vartheta},~[\pi+\bar{\rho}+\bar{\vartheta},~\alpha]^{nLie}]^{nLie}\\
&=[[\pi+\bar{\rho}+\bar{\vartheta},~\pi+\bar{\rho}+\bar{\vartheta}]^{nLie},~ \alpha]^{nLie}-[\pi+\bar{\rho}+\bar{\vartheta},~[\pi+\bar{\rho}+\bar{\vartheta},~ \alpha]^{nLie}]^{nLie}
\end{align*}
So, we have
\begin{align*}
2[\pi+\bar{\rho}+\bar{\vartheta},~[\pi+\bar{\rho}+\bar{\vartheta},~ \alpha]^{nLie}]^{nLie}&=[[\pi+\bar{\rho}+\bar{\vartheta},~\pi+\bar{\rho}+\bar{\vartheta}]^{nLie},~ \alpha]^{nLie}\\
\Rightarrow [\pi+\bar{\rho}+\bar{\vartheta},~[\pi+\bar{\rho}+\bar{\vartheta},~ \alpha]^{nLie}]^{nLie}&=\frac{1}{2} [[\pi+\bar{\rho}+\bar{\vartheta},~\pi+\bar{\rho}+\bar{\vartheta}]^{nLie},~ \alpha]^{nLie}\\
&=0.
\end{align*}
Hence the proof is complete.
\end{proof}

Note that
\begin{eqnarray*}
  \bullet ~\mathcal{Z}_{p}(\mathfrak{g};~V)&=&\{\alpha\in C_{>}^{p-1}(\mathfrak{g}\oplus V,~V )~|~ d(\alpha)=0\}\\
 &=& the~ set ~of ~p-cocycles.\\
  \bullet ~\mathcal{B}_{p}(\mathfrak{g};~V)&=&\{\alpha\in C_{>}^{p-1}(\mathfrak{g}\oplus V,~V )~|~ \exists~ \beta \in C_{>}^{p-2}(\mathfrak{g}\oplus V,~V ) ~such ~that ~\alpha=d(\beta)\}\\
 &=& the~ set ~of ~p-coboundaries.
\end{eqnarray*}
 Also the $p$-th cohomology group  $\mathcal{H}_{p}(\mathfrak{g};~V)$ is defined by $\mathcal{Z}_{p}(\mathfrak{g};~V)/\mathcal{B}_{p}(\mathfrak{g};~V).$
\begin{prop}\label{21}
    ~~\\
\rm
Let $\mathfrak{g}$ be a $n$-Lie algebra, we have $\mathcal{B}_{p}(\mathfrak{g};~V)\subset \mathcal{Z}_{p}(\mathfrak{g};~V).$
\end{prop}

\begin{proof}
      ~~\\
\rm
 By using Proposition \ref{m20}, Let $\alpha \in \mathcal{B}_{p}(\mathfrak{g};~V)$ such that $\alpha \in C_{>}^{p-1}(\mathfrak{g}\oplus V,~V )$ then
there exists $\beta \in  C_{>}^{p-2}(\mathfrak{g}\oplus V,~V )$ such that
                                     $$ \alpha=d(\beta),$$

hence
$$  d(\alpha)=d^{2}(\beta)=0, $$
this implies that  $\alpha\in \mathcal{Z}_{p}(\mathfrak{g};~V),$ which finishes the proof.
\end{proof}

Now we show that, how the cohomology complexes $\mathcal{H}_{p}(\mathfrak{g};~V)$ and  $ H^{p}(\mathfrak{g};~V)$ of the  $n$-Lie algebra  are interrelated. For this consider the following proposition.
\begin{prop}
        ~~\\
\rm
We can found a forgetful map from $\mathcal{H}_{p}(\mathfrak{g};~V)$ to $ H^{p}(\mathfrak{g};~V)$.
\end{prop}
\begin{proof}
 ~~\\
\rm
Obviously, we have $C^{p}(\mathfrak{g};~V)\subseteq C^{p}_{>}(\mathfrak{g} \oplus V;~V)$. For any $\alpha \in  C^{p}(\mathfrak{g};~V),~\mathfrak{X}_{i}\in \wedge^{n-1} \mathfrak{g} $ and $z\in \mathfrak{g},$ we have
$$ d(\alpha)=[\pi+\bar{\rho}+\bar{\vartheta},~\alpha]^{nLie}=(-1)^{p} \pi+\bar{\rho}+\bar{\vartheta}\circ\alpha - \alpha\circ\pi+\bar{\rho}+\bar{\vartheta}. $$

\begin{enumerate}
  \item Let $J=\{j_{1},~\cdots,~j_{p+1}\}\subset N \triangleq \{1,~2,~\cdots,~p+2\}$ and $I=\{i_{1}\}=N/J$ we have
  $\pi+\bar{\rho}+\bar{\vartheta}\circ\alpha(\mathfrak{X}_{1},~\cdots,~\mathfrak{X}_{p+1},~z)$
\begin{eqnarray*}
 &=& \displaystyle\sum_{J,~j_{p+1}<p+2}\sum_{s=1}^{n-1 } (-1)^{(J,~I)}(-1)^{k} \\&& \pi+\bar{\rho}+\bar{\vartheta}(x^{1}_{j_{p+1}}\wedge ~\cdots\wedge~x^{s-1}_{j_{p+1}}\wedge ~\alpha(\mathfrak{X}_{j_{1}}\wedge ~\cdots
\wedge ~\mathfrak{X}_{j_{p}}\wedge ~x^{s}_{j_{p+1}})\wedge ~x^{s+1}_{j_{p+1}}\wedge ~\cdots\wedge~x^{n-1}_{j_{p+1}},~z)\\&&-\displaystyle\sum_{J,~j_{p+1}=p+2} (-1)^{(J,~I)}
\pi+\bar{\rho}+\bar{\vartheta}(\mathfrak{X}_{i_{1}},~\alpha(  \mathfrak{X}_{j_{1}} ,~  \cdots,~\mathfrak{X}_{j_{p}},~z))\\
&=& \displaystyle\sum_{J,~j_{p+1}<p+2}\sum_{s=1}^{n-1 }(-1)^{n-s} \\&&\rho(x^{1}_{j_{p+1}},~\cdots, x^{s-1}_{j_{p+1}},~\hat{x^{s}_{j_{p+1}}},~x^{s+1}_{j_{p+1}},~\cdots, ~x^{n-1}_{j_{p+1}},~z)\alpha(\mathfrak{X}_{j_{1}},~\cdots
,~\mathfrak{X}_{j_{p}},~x^{s}_{j_{p+1}})\\&&-\displaystyle\sum_{J,~j_{p+1}=p+2} (-1)^{p+1-(i_{1}+1)}
\rho(\mathfrak{X}_{i_{1}})\alpha(  \mathfrak{X}_{j_{1}} ,~  \cdots,~\mathfrak{X}_{j_{p}},~z))\\
&=&  \displaystyle\sum_{J,~j_{p+1}<p+2}\sum_{s=1}^{n-1 }(-1)^{n-s} \\&&\rho(x^{1}_{j_{p+1}},~\cdots, x^{s-1}_{j_{p+1}},~\hat{x^{s}_{j_{p+1}}},~x^{s+1}_{j_{p+1}},~\cdots, ~x^{n-1}_{j_{p+1}},~z)\alpha(\mathfrak{X}_{j_{1}},~\cdots
,~\mathfrak{X}_{j_{p}},~x^{s}_{j_{p+1}})\\&&+\displaystyle\sum_{s=1}^{p+1}(-1)^{p-s+1}
\rho(\mathfrak{X}_{s})\alpha(  \mathfrak{X}_{1} ,~  \cdots,~ \hat{\mathfrak{X}_{s}} ,~\cdots,~\mathfrak{X}_{p+1},~z)\\
&=& \sum_{s=1}^{n-1 }(-1)^{n-s} \\&&\rho(x^{1}_{p+1},~\cdots,~ x^{s-1}_{p+1},~\hat{x^{s}_{p+1}},~x^{s+1}_{p+1},~\cdots, ~x^{n-1}_{p+1},~z)\alpha(\mathfrak{X}_{1},~\cdots
,~\mathfrak{X}_{p},~x^{s}_{p+1})\\&&+\displaystyle\sum_{s=1}^{p+1}(-1)^{p-s+1}
\rho(\mathfrak{X}_{s})\alpha(  \mathfrak{X}_{1} ,~  \cdots,~ \hat{\mathfrak{X}_{s}} ,~\cdots,~\mathfrak{X}_{p+1},~z).
\end{eqnarray*}
Then
\begin{equation*}
\begin{aligned}
(-1)^{p}\pi+\bar{\rho}+\bar{\vartheta}\circ\alpha&(\mathfrak{X}_{1},~\cdots,~\mathfrak{X}_{p+1},~z)\\& =  \displaystyle \sum_{s=1}^{n-1 } (-1)^{p+n-s}\\&\rho(x^{1}_{p+1},~\cdots,~\hat{x^{s}_{p+1}},~\cdots,~x^{n-1}_{p+1},~z)\alpha(\mathfrak{X}_{1},~\cdots
,~\mathfrak{X}_{p},~x^{s}_{p+1})\\&+\displaystyle\sum_{s=1}^{p+1}(-1)^{s+1}
\rho(\mathfrak{X}_{s})\alpha(  \mathfrak{X}_{1} ,~  \cdots,~ \hat{\mathfrak{X}_{s}} ,~\cdots,~\mathfrak{X}_{p+1},~z).
\end{aligned}
\end{equation*}

\item Let $J=\{j_{1},~j_{2}\}_{j_{1}<j_{2}} \subset N\triangleq \{1,~\cdots,~p+2\}$ and $I=\{i_{1},~\cdots,~i_{p}\}=N/J$ we have\\
\end{enumerate}
$\alpha\circ\pi+\bar{\rho}+\bar{\vartheta} (\mathfrak{X}_{1},~\cdots,~\mathfrak{X}_{p+1},~z)$
\begin{eqnarray*}
&=&\displaystyle\sum_{J,~j_{2}<p+2}\sum_{s=1}^{n-1 } (-1)^{(J,~I)} (-1)^{k}\\&&
           \alpha  ( \mathfrak{X}_{i_{1}},~\cdots,~\mathfrak{X}_{i_{k}},~ x^{1}_{j_{2}}\wedge\cdots\wedge x^{s-1}_{j_{2}}\wedge\pi+\bar{\rho}+\bar{\vartheta}( \mathfrak{X}_{j_{1}},~x^{s}_{j_{2}})\wedge x^{s+1}_{j_{2}}\wedge  \cdots \wedge x^{n-1}_{j_{2}},~\mathfrak{X}_{i_{k+1}},~ \cdots ,~ \mathfrak{X}_{i_{p-1}},~z)\\
&&+\displaystyle\sum_{J,~j_{2}=p+2}  (-1)^{(J,~I)} (-1)^{p} \alpha(\mathfrak{X}_{i_{1}},~\cdots,~\mathfrak{X}_{i_{p}},~\pi+\bar{\rho}+\bar{\vartheta}(\mathfrak{X}_{j_{1}},~z))\\
&=&\displaystyle\sum_{J,~j_{2}<p+2}\sum_{s=1}^{n-1 } (-1)^{(J,~I)} (-1)^{k}\\&&
           \alpha  ( \mathfrak{X}_{i_{1}},~\cdots,~\mathfrak{X}_{i_{k}},~ x^{1}_{j_{2}}\wedge\cdots\wedge x^{s-1}_{j_{2}}\wedge \pi( \mathfrak{X}_{j_{1}},~x^{s}_{j_{2}})\wedge x^{s+1}_{j_{2}}\wedge  \cdots \wedge x^{n-1}_{j_{2}},~\mathfrak{X}_{i_{k+1}},~ \cdots ,~ \mathfrak{X}_{i_{p-1}},~z)\\
&&+\displaystyle\sum_{J,~j_{2}=p+2}  (-1)^{(J,~I)} (-1)^{p} \alpha(\mathfrak{X}_{i_{1}},~\cdots,~\mathfrak{X}_{i_{p}},~\pi(\mathfrak{X}_{j_{1}},~z))\\
&=& \displaystyle\sum_{J,~1\leq j_{1}< j_{2}<p+2} (-1)^{(J,~I)} (-1)^{k}
 \alpha  ( \mathfrak{X}_{i_{1}},~\cdots,~\mathfrak{X}_{i_{k}},~[\mathfrak{X}_{j_{1}},~ \mathfrak{X}_{j_{2}} ]_{F},~\mathfrak{X}_{i_{k+1}},~ \cdots ,~ \mathfrak{X}_{i_{p-1}},~z)\\
&&+\displaystyle\sum_{J,~j_{2}=p+2}  (-1)^{(J,~I)} (-1)^{p} \alpha(\mathfrak{X}_{i_{1}},~\cdots,~\mathfrak{X}_{i_{p}},~\pi(\mathfrak{X}_{j_{1}},~z)).\\
&=& \displaystyle\sum_{J,~1\leq j_{1}< j_{2}<p+2} (-1)^{k+1-j_{1}}
 \alpha  ( \mathfrak{X}_{i_{1}},~\cdots,~\mathfrak{X}_{i_{k}},~[\mathfrak{X}_{j_{1}},~ \mathfrak{X}_{j_{2}} ]_{F},~\mathfrak{X}_{i_{k+1}},~ \cdots ,~ \mathfrak{X}_{i_{p-1}},~z)\\
&&+\displaystyle\sum_{J,~j_{2}=p+2} (-1)^{p+1-j_{1}} \alpha(\mathfrak{X}_{i_{1}},~\cdots,~\mathfrak{X}_{i_{p}},~\pi(\mathfrak{X}_{j_{1}},~z)).\\
&=& \displaystyle\sum_{J,~1\leq j_{1}< j_{2}<p+2} (-1)^{1-j_{1}}
 \alpha  ( \mathfrak{X}_{i_{1}},~\cdots,~\mathfrak{X}_{i_{k}},~[\mathfrak{X}_{j_{1}},~ \mathfrak{X}_{j_{2}} ]_{F},~\mathfrak{X}_{i_{k+1}},~ \cdots ,~ \mathfrak{X}_{i_{p-1}},~z)\\
&&+\displaystyle\sum_{J,~j_{2}=p+2}  (-1)^{1-j_{1}} \alpha(\mathfrak{X}_{i_{1}},~\cdots,~\mathfrak{X}_{i_{p}},~\pi(\mathfrak{X}_{j_{1}},~z)).\\
&=&- \displaystyle\sum_{J,~1\leq j_{1}< j_{2}<p+2} (-1)^{j_{1}}
 \alpha  ( \mathfrak{X}_{i_{1}},~\cdots,~\mathfrak{X}_{i_{k}},~[\mathfrak{X}_{j_{1}},~ \mathfrak{X}_{j_{2}} ]_{F},~\mathfrak{X}_{i_{k+1}},~ \cdots ,~ \mathfrak{X}_{i_{p-1}},~z)\\
&&-\displaystyle\sum_{J,~j_{2}=p+2}  (-1)^{j_{1}} \alpha(\mathfrak{X}_{i_{1}},~\cdots,~\mathfrak{X}_{i_{p}},~\pi(\mathfrak{X}_{j_{1}},~z)).
\end{eqnarray*}
Hence, we obtain the result that  $$ d(\alpha)(\mathfrak{X}_{1},~\cdots,~\mathfrak{X}_{p+1},~z) =\delta_{\rho}(\alpha)(\mathfrak{X}_{1},~\cdots,~\mathfrak{X}_{p+1},~z),~~~~\forall \alpha\in C^{p}(\mathfrak{g},~V),~ p\geq0.$$
From the above expression, we deduce that coboundary operator $\delta_{\rho}(\alpha)$ and the coboundary operator provided by Eq.\eqref{6} are similar. Hence, the natural projection from $C^{p}_{>}(\mathfrak{g} \oplus V,~V) $ to $C^{p}(\mathfrak{g},~V)$ induces a forgetful map from $\mathcal{H}_{p}(\mathfrak{g};~V)$ to $ H^{p}(\mathfrak{g};~V)$
\end{proof}
Now, we characterize the low dimensional cocycles of  $n$-Lie algebra   $\mathfrak{g} \oplus V$  with values in  $V.$  
\begin{prop}
           ~~\\
\rm
Let $\alpha\in Hom(\mathfrak{g};~V).$ Then $\alpha$ will said to be $1$-cocycle if and only if for $ \mathfrak{X}_{1}=(x_{1},~\cdots,~x_{n-1})\in \wedge^{n-1}\mathfrak{g},~x_{n}\in \mathfrak{g}$, we have the following identity\\
$0=-\alpha([x_{1},~\cdots,~x_{n-1},~x_{n}]_{\mathfrak{g}})+\rho(x_{1},~\cdots,~x_{n-1})(\alpha(x_{n}))+\displaystyle\sum^{n-1}_{i=1}(-1)^{n-i}
\rho(x_{1},~\cdots,~\hat{x_{i}},~\cdots,~x_{n-1},~x_{n})(\alpha(x_{i})).$
\end{prop}

\begin{proof}
           ~~\\
\rm
Let $\alpha\in Hom(\mathfrak{g};~V),$ we have
\begin{eqnarray*}
  d(\alpha)(\mathfrak{X}_{1},~x_{n}) &=&\delta_{\rho}(\alpha)(\mathfrak{X}_{1},~x_{n})\\&
=&-\alpha([\mathfrak{X}_{1},~x_{n}]_{\mathfrak{g}})+\rho(\mathfrak{X}_{1})(\alpha(x_{n}))+\displaystyle\sum_{i=1}^{n-1}(-1)^{n-i}\rho(x_{1},~
\cdots,~\hat{x_{i}},~\cdots,~x_{n-1},~x_{n})(\alpha(x_{i}))\\
&=&-\alpha([x_{1},~\cdots,~x_{n-1},~x_{n}]_{\mathfrak{g}})+\rho(x_{1},~\cdots,~x_{n-1})(\alpha(x_{n}))+\displaystyle\sum_{i=1}^{n-1}(-1)^{n-i}\\&&\rho(x_{1},~
\cdots,~\hat{x_{i}},~\cdots,~x_{n-1},~x_{n})(\alpha(x_{i})).\end{eqnarray*}
\end{proof}
\begin{prop}~~\\
\rm
Let $\beta_{1} \in Hom(\wedge^{n-1}V\wedge \mathfrak{g},~V),~\beta_{2} \in Hom(\wedge^{n-1}\mathfrak{g}\wedge V,~V)$ and $\beta_{3} \in Hom(\wedge^{n}\mathfrak{g},~V)$, a 2-cochain $\beta_{1} +\beta_{2} +\beta_{3} \in C_{>}^{1}(\mathfrak{g}\oplus V,~V)$ is a 2-cocycle if and only if for all $x_{i}\in \mathfrak{g},~ u_{i}\in V$ and $h_{i} \in V,$ the following identities hold:
\begin{eqnarray}
  0 &=& -\displaystyle\sum_{s=1}^{n-1} (-1)^{n-s}\rho(y_{1},~\cdots,~\hat{y_{s}},~\cdots,~y_{n-1},~z)(\beta_{3} (x_{1},~\cdots,~x_{n-1},~y_{s}))
\nonumber\\&&\nonumber+\rho(x_{1},~\cdots,~x_{n-1})(\beta_{3}(y_{1},~\cdots,~y_{n-1},~z))- \rho(y_{1},~\cdots,~y_{n-1})(\beta_{3} (x_{1},~\cdots,~x_{n-1},~z))\\&&\nonumber -\displaystyle\sum_{s=1}^{n-1} \beta_{3} (y_{1},~\cdots,~y_{s-1},~[x_{1},~\cdots,~x_{n-1},~y_{s}]_{\mathfrak{g}},~y_{s+1},~\cdots,~y_{n-1},~z)
+\\&&\beta_{3}(x_{1},~\cdots,~x_{n-1},~[y_{1},~\cdots,~y_{n-1},~z]_{\mathfrak{g}})
-\beta_{3}(y_{1},~\cdots,~y_{n-1},~[x_{1},~\cdots,~x_{n-1},~z]_{\mathfrak{g}}),\label{23}
\\0&=&-(-1)^{n-2}\vartheta(y_{2})(\beta_{3}(x_{1},~\cdots,~ x_{n-1},~y_{1}),~h_{3},~\cdots,~h_{n-1},~h_{n})\nonumber\\&&-(-1)^{n-1}\vartheta(y_{1})(\beta_{3}(x_{1},~\cdots,~ x_{n-1},~y_{2}),~h_{3},~\cdots,~h_{n-1},~h_{n}),\label{24}\\
0&=&(-1)^{n-1}\vartheta(x_{1})(u_{2},~\cdots,~u_{n-1},~\beta_{3}(y_{1},~\cdots,~ y_{n-1},~z)),\label{25}\\\nonumber
0&=& \rho(x_{1},~\cdots,~x_{n-1})(\beta_{2}(y_{1},~\cdots,~y_{n-1},~h_{1}))\nonumber\\&&-\rho(y_{1},~\cdots,~y_{n-1})(\beta_{2}(x_{1},~\cdots,~x_{n-1},~h_{1}))
\nonumber \\&&- \displaystyle \sum_{s=1}^{n-1}\beta_{2}(y_{1},~\cdots,~y_{s-1},~[x_{1},~\cdots,~x_{n-1},~y_{s}]_{\mathfrak{g}},~y_{s+1},~\cdots,~y_{n-1},~h_{1})\nonumber\\
&&+\beta_{2}(x_{1},~\cdots,~x_{n-1},~\rho(y_{1},~\cdots,~y_{n-1})(h_{1}))\nonumber\\
&&-\beta_{2}(y_{1},~\cdots,~y_{n-1},~\rho(x_{1},~\cdots,~x_{n-1})(h_{1})),\label{26}
\end{eqnarray}
\begin{eqnarray}
0&=&\rho(y_{1},~\cdots,~y_{n-2},~y_{n})(\beta_{2}(x_{1},~\cdots,~x_{n-1},~h_{n-1}))\nonumber\\&&+
\rho(x_{1},~\cdots,~x_{n-1})(\beta_{2}(y_{1},~\cdots,~y_{n-2},~h_{n-1},~y_{n}))\nonumber\\&&
-\displaystyle\sum_{s=1}^{n-2}\beta_{2}(y_{1},~\cdots,~[x_{1},~\cdots,~x_{n-1},~y_{s}]_{\mathfrak{g}},~\cdots,~y_{n-2},~h_{n-1},~y_{n})\nonumber\\&&
-\beta_{2}(y_{1},~\cdots,~y_{n-2},~\rho(x_{1},~\cdots,~x_{n-1})(h_{n-1}),~y_{n})\nonumber\\&&
+\beta_{2}(x_{1},~\cdots,~x_{n-1},~\rho(y_{1},~\cdots,~y_{n-2},~y_{n})(h_{n-1}))\nonumber\\&&
-\beta_{2}(y_{1},~\cdots,~y_{n-2},~h_{n-1},~[x_{1},~\cdots,~x_{n-1},~y_{n}]_{\mathfrak{g}}),\label{27}\\
0&=&-\displaystyle\sum_{s=2}^{n-1}(-1)^{n-1}\vartheta(y_{1})(h_{2},~\cdots,~h_{s-1},~
\beta_{2}(x_{1},~\cdots,~x_{n-1},~h_{s}),~h_{s+1},~\cdots,~h_{n-1},~h_{n})\nonumber\\
 &&-(-1)^{n-1}\vartheta(y_{1})(h_{2},~\cdots,~h_{n-1},~\beta_{2}(x_{1},~\cdots,~x_{n-1},~h_{n}))\nonumber\\
&&+(-1)^{n-1}\beta_{2}(x_{1},~\cdots,~x_{n-1},~\vartheta(y_{1})(h_{2},~\cdots,~h_{n-1},~h_{n})),\label{28}\\
0&=&-\displaystyle\sum_{s=1}^{n-1}(-1)^{n-s}\rho(y_{1},~\cdots,~\hat{y_{s}},~\cdots,~y_{n-1},~z)(\beta_{2}(x_{1},~\cdots,~x_{n-2},~u_{1},~y_{s})\nonumber\\&&
-\rho(y_{1},~\cdots,~y_{n-1})(\beta_{2}(x_{1},~\cdots,~x_{n-2},~u_{1},~z))\nonumber\\&&
+\displaystyle\sum_{s=1}^{n-1}\beta_{2}(y_{1},~\cdots,~y_{s-1},~\rho(x_{1},~\cdots,~x_{n-2},~y_{s})(u_{1}),~y_{s+1},~\cdots,~y_{n-1},~z)\nonumber\\
&&+\beta_{2}(x_{1},~\cdots,~x_{n-2},~u_{1},~[y_{1},~\cdots,~y_{n-1},~z]_{\mathfrak{g}})\nonumber\\
&&+\beta_{2}(y_{1},~\cdots,~y_{n-1},~\rho(x_{1},~\cdots,~x_{n-2},~z)(u_{1})),\label{29}\\
0&=&-(-1)^{n-2}\vartheta(y_{2})(x_{1},~\cdots,~x_{n-2},~u_{1},~y_{1}),~h_{3},~\cdots,~h_{n-1},~h_{n})\nonumber\\&&
  -(-1)^{n-1}\vartheta(y_{1})(x_{1},~\cdots,~x_{n-2},~u_{1},~y_{2}),~h_{3},~\cdots,~h_{n-1},~h_{n}),\label{30}\\
0&=&(-1)^{n-1}\vartheta(x_{1})(u_{2},~\cdots,~u_{n-1},~\beta_{2}(y_{1},~\cdots,~y_{n-1},~h_{n}))\nonumber\\&&-(-1)^{n-1}
\beta_{2}(y_{1},~\cdots,~y_{n-1},~\vartheta(x_{1})(u_{2},~\cdots,~u_{n-1},~h_{n})),\label{31}\\
0&=&(-1)^{n-1}\vartheta(x_{1})(u_{2},~\cdots,~u_{n-1},~\beta_{2}(y_{1},~\cdots,~y_{n-2},~h_{n-1},~y_{n}))\nonumber\\&&-(-1)^{n-1}
\beta_{2}(y_{1},~\cdots,~y_{n-2},~\vartheta(x_{1})(u_{2},~\cdots,~u_{n-1},~h_{n-1}),~y_{n}),\label{32}\\
0&=&-\displaystyle\sum_{s=1}^{n-1}\beta_{2}(y_{1},~\cdots,~\vartheta(y_{s})(u_{1},~\cdots,~u_{n-1}),~\cdots
,~y_{n})\nonumber\\&&- \beta_{2}(y_{1},~\cdots,~y_{n-1},~\vartheta(y_{n})(u_{1},~\cdots,~u_{n-1})),\label{33}\\
0&=&-\displaystyle\sum_{s=1}^{n-1}\vartheta(y_{n})(h_{1},~\cdots,~\beta_{2}(x_{1},~\cdots,~x_{n-1},~h_{i}),~\cdots,~h_{n-1})
\nonumber\\&&+\beta_{2}(x_{1},~\cdots,~x_{n-1},~\vartheta(y_{n})(h_{1},~\cdots,~h_{n-1})),\label{34}\\\nonumber
0&=&\rho(x_{1},~\cdots,~x_{n-1})(\beta_{1}(y_{1},~h_{2},~\cdots,~h_{n-1},~h_{n}))\nonumber\\&&-
\beta_{1}([x_{1},~\cdots,~x_{n-1},~y_{1}]_{\mathfrak{g}},~h_{2},~\cdots,~h_{n-1},~h_{n})\nonumber\\&&
-\displaystyle\sum_{s=2}^{n-1}\beta_{1}(y_{1},~h_{2},~\cdots,~h_{s-1},~\rho(x_{1},~\cdots,~x_{n-1})(h_{s}),~h_{s+1},~\cdots,~h_{n-1},~h_{n})
\nonumber\\&&-\beta_{1}(y_{1},~h_{2},~\cdots,~h_{n-1},~\rho(x_{1},~\cdots,~x_{n-1})(h_{n})),\label{35}
\end{eqnarray}
\begin{eqnarray}
0&=&-\displaystyle\sum_{s=1}^{n-1}\vartheta(y_{n})(h_{1},~\cdots,~\beta_{1}(x_{1},~\cdots,~x_{n-1},~h_{i}),~\cdots,~h_{n-1})\nonumber\\&&
+\rho(x_{1},~\cdots,~x_{n-1})(\beta_{1}(h_{1},~\cdots,~h_{n-1},~y_{n}))\nonumber\\
&&-\displaystyle\sum_{s=1}^{n-1}\beta_{1}(h_{1},~\cdots,~\rho(x_{1},~\cdots,~x_{n-1})(h_{i}),~\cdots,~h_{n-1},~y_{n})\nonumber\\
&&-\beta_{1}(h_{1},~\cdots,~h_{n-1},~[x_{1},~\cdots,~x_{n-1},~y_{n}]_{\mathfrak{g}}),\label{36}\\
0&=&-\rho(y_{1},~\cdots,~y_{n-1})(\beta_{1}(x_{1},~u_{2},~\cdots,~u_{n-1},~h_{n}))\nonumber\\&&
+\beta_{1}(x_{1},~u_{2},~\cdots,~u_{n-1},~\rho(y_{1},~\cdots,~y_{n-1})(h_{n})),\label{37}\\
0&=&- \displaystyle\sum_{s=2}^{n-1}  \vartheta(y_{1})(h_{2},~\cdots,~h_{s-1},~\beta_{1}(x_{1},~u_{2},~\cdots,~u_{n-1},~h_{s}),~h_{s+1},~\cdots,~h_{n-1},~h_{n})\nonumber\\&&
+(-1)^{n-1}\vartheta(x_{1})(u_{2},~\cdots,~u_{n-1},~\beta_{1}(y_{1},~h_{2},~\cdots,~h_{n-1},~h_{n}))\nonumber\\&&
-(-1)^{n-1}\vartheta(y_{1})(h_{2},~\cdots,~h_{n-1},~\beta_{1}(x_{1},~u_{2},~\cdots,~u_{n-1},~h_{n}))\nonumber\\
&&- \displaystyle\sum_{s=2}^{n-1} (-1)^{n-1}\beta_{1} (y_{1},~h_{2},~\cdots,~h_{s-1},~\vartheta(x_{1})(u_{2},~\cdots,~u_{n-1},~h_{s}),~h_{s+1},~\cdots,~h_{n-1},~h_{n})\nonumber\\&&
+(-1)^{n-1}\beta_{1}(x_{1},~u_{2},~\cdots,~u_{n-1},~\vartheta(y_{1})(h_{2},~\cdots,~h_{n-1},~h_{n}))\nonumber\\&&
-(-1)^{n-1}\beta_{1}(y_{1},~h_{2},~\cdots,~h_{n-1},~\vartheta(x_{1})(u_{2},~\cdots,~u_{n-1},~h_{n})),\label{38}\\
0&=&- \displaystyle\sum_{s=1}^{n-1}\vartheta(y_{n}) (h_{1},~\cdots,~\beta_{1}(x_{1},~u_{2},~\cdots,~u_{n-1},~h_{i}),~\cdots,~h_{n-1})\nonumber\\
&&+(-1)^{n-1}\vartheta(x_{1})(u_{2},~\cdots,~u_{n-1},~\beta_{1}(h_{1},~\cdots,~h_{n-1},~y_{n}))\nonumber\\
&&-\displaystyle\sum_{s=1}^{n-1}(-1)^{n-1}\beta_{1}(h_{1},~\cdots,~\vartheta(x_{1})(u_{2},~\cdots,~u_{n-1},~h_{i}),~\cdots,~h_{n-1},~y_{n})
\nonumber\\&&+\beta_{1}(x_{1},~u_{2},~\cdots,~u_{n-1},~\vartheta(y_{n})(h_{1},~\cdots,~h_{n-1})),\label{39}\\
0&=&-\displaystyle\sum_{s=1}^{n-1} (-1)^{n-s}\rho(y_{1},~\cdots,~y_{s-1},~\hat{y_{s}},~y_{s+1},~\cdots,~y_{n-1},~z)(\beta_{1}(u_{1},~\cdots,~u_{n-1},~y_{s}))
\nonumber\\&&- \rho(y_{1},~\cdots,~y_{n-1})(\beta_{1}(u_{1},~\cdots,~u_{n-1},~z))\nonumber\\&&+\beta_{1}(u_{1},~\cdots,~u_{n-1},~[y_{1},~\cdots
,~y_{n-1},~z]_{\mathfrak{g}}),\label{40}\\
0&=&-(-1)^{n-2}\vartheta(y_{2})(\beta_{1}(u_{1},~\cdots,~u_{n-1},~y_{1}),~h_{3},~\cdots,~h_{n-1},~h_{n})\nonumber\\&&
-(-1)^{n-1}\vartheta(y_{1})(\beta_{1}(u_{1},~\cdots,~u_{n-1},~y_{2}),~h_{3},~\cdots,~h_{n-1},~h_{n})\nonumber\\&&
 -\beta_{1}(\vartheta(y_{1})(u_{1},~\cdots,~u_{n-1}),~y_{2},~h_{3},~\cdots,~h_{n})\nonumber\\&&
-\beta_{1}(y_{1},~\vartheta(y_{2})(u_{1},~\cdots,~u_{n-1}),~h_{3},~\cdots,~h_{n}),\label{41}\\\nonumber
0&=&- \vartheta(y_{n})(\beta_{1}(u_{1},~\cdots,~u_{n-1},~y_{1}),~h_{2},~\cdots,~h_{n-1})\nonumber\\&&
-(-1)^{n-1}\vartheta(y_{1})(h_{2},~\cdots,~h_{n-1},~\beta_{1}(u_{1},~\cdots,~u_{n-1},~y_{n})\nonumber\\&&
-\beta_{1}( \vartheta(y_{1})(u_{1},~\cdots,~u_{n-1}),~h_{2},~\cdots,~h_{n-1},~y_{n})\nonumber\\&&-
\beta_{1}(y_{1},~h_{2},~\cdots,~h_{n-1},~\vartheta(y_{n})(u_{1},~\cdots,~u_{n-1})).\label{42}
\end{eqnarray}
\end{prop}
\begin{proof}
 ~~\\
\rm
For $\beta_{3}\in Hom(\wedge^{n}\mathfrak{g},~V)$, we have
\begin{description}
 \item[$\ast$] $d(\beta_{3})(x_{1}~\wedge~\cdots~\wedge ~x_{n-1},~y_{1}~\wedge~\cdots~\wedge ~y_{n-1},~z)=\\
-\displaystyle\sum_{s=1}^{n-1} (-1)^{n-s}\rho(y_{1},~\cdots,~\hat{y_{s}},~\cdots,~y_{n-1},~z)(\beta_{3}(x_{1},~\cdots,~ x_{n-1},~y_{s}))\\
+\rho(x_{1},~\cdots,~ x_{n-1})(\beta_{3}(y_{1},~\cdots,~y_{n-1},~z))-\rho(y_{1},~\cdots,~ y_{n-1})(\beta_{3}(x_{1},~\cdots,~x_{n-1},~z))\\
- \displaystyle\sum_{s=1}^{n-1} \beta_{3}(y_{1},~\cdots,~y_{s-1},~[x_{1},~\cdots,~ x_{n-1},~y_{s}]_{\mathfrak{g}},~y_{s+1},~\cdots,~y_{n-1},~z)\\
+\beta_{3}(x_{1},~\cdots,~ x_{n-1},~[y_{1},~\cdots,~y_{n-1},~z]_{\mathfrak{g}})-\beta_{3}(y_{1},~\cdots,~ y_{n-1},~[x_{1},~\cdots,~x_{n-1},~z]_{\mathfrak{g}}),$
 \item[$\ast$]$d(\beta_{3})(x_{1},~\cdots,~ x_{n-1},~y_{1},~y_{2},~h_{3},~\cdots,~h_{n-1},~h_{n}) =\\
-(-1)^{n-2}\vartheta(y_{2})(\beta_{3}(x_{1},~\cdots,~ x_{n-1},~y_{1}),~h_{3},~\cdots,~h_{n-1},~h_{n})\\-(-1)^{n-1}\vartheta(y_{1})(\beta_{3}(x_{1},~\cdots,~ x_{n-1},~y_{2}),~h_{3},~\cdots,~h_{n-1},~h_{n}),$
\item [$\ast$]$d(\beta_{3})(x_{1}~\wedge~u_{2}~\wedge\cdots~\wedge ~u_{n-1},~y_{1},~\cdots,~ y_{n-1},~z)\\
= (-1)^{n-1}\vartheta(x_{1})(u_{2},~\cdots,~u_{n-1},~\beta_{3}(y_{1},~\cdots,~ y_{n-1},~z)).$\\
For $\beta_{2}\in Hom(\wedge^{n-1}\mathfrak{g}\wedge V,~V)$ we have
\item[$\ast$]$d(\beta_{2})(x_{1}~\wedge~\cdots~\wedge ~x_{n-1},~y_{1}~\wedge~\cdots~\wedge ~y_{n-1},~h_{1})=\\
\rho(x_{1},~\cdots,~x_{n-1})(\beta_{2}(y_{1},~\cdots,~y_{n-1},~h_{1}))\\-\rho(y_{1},~\cdots,~y_{n-1})(\beta_{2}(x_{1},~\cdots,~x_{n-1},~h_{1}))
 \\- \displaystyle \sum_{s=1}^{n-1}\beta_{2}(y_{1},~\cdots,~y_{s-1},~\pi(x_{1},~\cdots,~x_{n-1},~y_{s}),~y_{s+1},~\cdots,~y_{n-1},~h_{1})\\
+\beta_{2}(x_{1},~\cdots,~x_{n-1},~\rho(y_{1},~\cdots,~y_{n-1})(h_{1}))\\
-\beta_{2}(y_{1},~\cdots,~y_{n-1},~\rho(x_{1},~\cdots,~x_{n-1})(h_{1})),$
\item[$\ast$] $d(\beta_{2})(x_{1}~\wedge~\cdots~\wedge ~x_{n-1},~y_{1}~\wedge~\cdots~\wedge ~y_{n-2},~h_{n-1},~y_{n})=\\
\rho(y_{1},~\cdots,~y_{n-2},~y_{n})(\beta_{2}(x_{1},~\cdots,~x_{n-1},~h_{n-1}))\\+\rho(x_{1},~\cdots,~x_{n-1})
(\beta_{2}(y_{1},~\cdots,~y_{n-2},~h_{n-1},~y_{n}))
 \\- \displaystyle \sum_{s=1}^{n-2}\beta_{2}(y_{1},~\cdots,~y_{s-1},~\pi(x_{1},~\cdots,~x_{n-1},~y_{s}),~y_{s+1},~\cdots,~y_{n-2},~h_{n-1},~y_{n})\\
-\beta_{2}(y_{1},~\cdots,~y_{n-2},~\rho(x_{1},~\cdots,~x_{n-1})(h_{n-1}),~y_{n})\\
+\beta_{2}(x_{1},~\cdots,~x_{n-1},~\rho(y_{1},~\cdots,~y_{n-2},~y_{n})(h_{n-1}))\\
-\beta_{2}(y_{1},~\cdots,~y_{n-2},~h_{n-1},~\pi(x_{1},~\cdots,~x_{n-1},~y_{n}))$,
\item[$\ast$] $d(\beta_{2})(x_{1}~\wedge~\cdots~\wedge ~x_{n-1},~y_{1}~\wedge ~h_{2}\wedge ~\cdots~\wedge ~h_{n-1},~h_{n})
=\\-\displaystyle\sum_{s=2}^{n-1}(-1)^{n-1}\vartheta(y_{1})(h_{2},~\cdots,~h_{s-1},~\beta_{2}(x_{1},~\cdots
,~x_{n-1},~h_{s}),~h_{s+1},~\cdots,~h_{n-1},~h_{n})\\
 -(-1)^{n-1}\vartheta(y_{1})(h_{2},~\cdots,~h_{n-1},~\beta_{2}(x_{1},~\cdots,~x_{n-1},~h_{n}))\\
+(-1)^{n-1}\beta_{2}(x_{1},~\cdots,~x_{n-1},~\vartheta(y_{1})(h_{2},~\cdots,~h_{n-1},~h_{n})),$
\item[$\ast$] $d(x_{1}~\wedge~\cdots~\wedge ~x_{n-2},~u_{1},~y_{1}~\wedge~\cdots~\wedge ~y_{n-1},~z)
 =\\-\displaystyle\sum_{s=1}^{n-1}(-1)^{n-s}\rho(y_{1},~\cdots,~\hat{y_{s}},~\cdots,~y_{n-1},~z)(\beta_{2}(x_{1},~\cdots,~x_{n-2},~u_{1},~y_{s})\\
-\rho(y_{1},~\cdots,~y_{n-1})(\beta_{2}(x_{1},~\cdots,~x_{n-2},~u_{1},~z))\\
+\displaystyle\sum_{s=1}^{n-1}\beta_{2}(y_{1},~\cdots,~y_{s-1},~\rho(x_{1},~\cdots,~x_{n-2},~y_{s})(u_{1}),~y_{s+1},~\cdots,~y_{n-1},~z
)\\
+\beta_{2}(x_{1},~\cdots,~x_{n-2},~u_{1},~\pi(y_{1},~\cdots,~y_{n-1},~z))\\
-\beta_{2}(y_{1},~\cdots,~y_{n-1},~\rho(x_{1},~\cdots,~x_{n-2},~z)(u_{1})),$
\item[$\ast$] $d(\beta_{2})(x_{1}~\wedge~\cdots~\wedge ~x_{n-2}~\wedge ~u_{1},~y_{1}~\wedge~y_{2}~\wedge~h_{3}~\wedge\cdots~\wedge ~h_{n-1},~h_{n})=\\-(-1)^{n-2}\vartheta(y_{2})(x_{1},~\cdots,~x_{n-2},~u_{1},~y_{1}),~h_{3},~\cdots,~h_{n-1},~h_{n})\\
  -(-1)^{n-1}\vartheta(y_{1})(x_{1},~\cdots,~x_{n-2},~u_{1},~y_{2}),~h_{3},~\cdots,~h_{n-1},~h_{n}),$
\item[$\ast$]$ d ( \beta_{2})(x_{1}~\wedge~u_{2}~\wedge ~\cdots~\wedge ~u_{n-1},~y_{1}~\wedge~\cdots~\wedge ~y_{n-1},~h_{n})=\\
(-1)^{n-1}\vartheta(x_{1})(u_{2},~\cdots,~u_{n-1},~\beta_{2}(y_{1},~\cdots, ~y_{n-1}, ~h_{n})\\
-(-1)^{n-1}\beta_{2}(y_{1},~\cdots, ~y_{n-1}, ~\vartheta(x_{1})(u_{2},~\cdots,~u_{n-1},~h_{n})),$
\item[$\ast$]$d(\beta_{2})(x_{1}~\wedge~ u_{2}~\wedge~\cdots~\wedge ~u_{n-1},~y_{1}~\wedge~\cdots~\wedge y_{n-2},~h_{n-1},~y_{n} ) =\\(-1)^{n-1}\vartheta(x_{1})(u_{2},~\cdots,~u_{n-1},~\beta_{2} ( y_{1},~\cdots,~ y_{n-2},~h_{n-1},~y_{n})\\
-(-1)^{n-1}\beta_{2}(y_{1},~ \cdots,~y_{n-2},~\vartheta(x_{1})(u_{2},~\cdots,~u_{n-1},~h_{n-1}),~y_{n}),$
\item[$\ast$] $ d ( \beta_{2})(u_{1}~\wedge~\cdots~\wedge ~u_{n-1},~y_{1}~\wedge~\cdots~\wedge ~y_{n-1},~z)=\\- \displaystyle\sum_{s=1}^{n-1}\beta_{2}(y_{1},~\cdots,~y_{s-1},~\vartheta(y_{s})(u_{1},~\cdots,~u_{n-1}),~y_{s+1},~\cdots,~y_{n-1},~z)\\
  -\beta_{2}(y_{1},~\cdots,~y_{n-1},~\vartheta(z)(u_{1},~\cdots,~u_{n-1})),$
\item[$\ast$] $d(\beta_{2})(x_{1}~\wedge~\cdots~\wedge ~x_{n-1},~h_{1}~\wedge~\cdots~\wedge h_{n-1},~y_{n} ) =\\-\displaystyle\sum_{s=1}^{n-1}\vartheta(y_{n})(h_{1},~\cdots,~\beta_{2} ( x_{1},~\cdots,~ x_{n-1},~h_{i}),~\cdots,~h_{n-1})\\
+\beta_{2}(x_{1},~ \cdots,~x_{n-1},~\vartheta(y_{n})(h_{1},~\cdots,~h_{n-1})).$\\
For $\beta_{1}\in Hom(\wedge^{n-1}V \wedge\mathfrak{g},~V)$, we have
\item [$\ast$] $ d(\beta_{1})(x_{1}~\wedge~\cdots~\wedge~x_{n-1},~y_{1}~\wedge~h_{2}~\wedge~\cdots~\wedge~h_{n-1},~h_{n})=\\
\rho(x_{1},~\cdots,~x_{n-1})(\beta_{1}(y_{1},~h_{2},~\cdots,~h_{n}))\\
-\beta_{1}([x_{1},~\cdots,~x_{n-1},~y_{1}]_{\mathfrak{g}},~h_{2},~\cdots,~h_{n-1},~h_{n})\\
-\displaystyle\sum_{s=2}^{n-1}\beta_{1}(y_{1},~h_{2},~\cdots,~h_{n-1},~\rho(x_{1},~\cdots,~x_{n-1})(h_{s}),~\cdots,~h_{n})\\
-\beta_{1}(y_{1},~h_{2},~\cdots,~h_{n-1},~\rho(x_{1},~\cdots,~x_{n-1})(h_{n})),$
\item[$\ast$] $ d(\beta_{1})(x_{1}~\wedge~\cdots~\wedge~x_{n-1},~h_{1}~\wedge~\cdots~\wedge~h_{n-1},~y_{n})=\\
-\displaystyle\sum_{s=1}^{n-1}
\vartheta(y_{n})(h_{1},~\cdots,~\beta_{1}(x_{1},~\cdots,~x_{n-1},~h_{i}),~\cdots ,~h_{n-1} )+\rho(x_{1},~\cdots,~x_{n-1})(\beta_{1}(h_{1},~\cdots,~h_{n-1},~y_{n}))\\
-\displaystyle\sum_{s=1}^{n-1}\beta_{1}(h_{1},~\cdots,~\rho(x_{1},~\cdots,~x_{n-1})(h_{s}),~\cdots,~h_{n-1},~y_{n})
-\beta_{1}(h_{1},~\cdots,~h_{n-1},~[x_{1},~\cdots,~x_{n-1},~y_{n}]_{\mathfrak{g}}),$
\item [$\ast$] $ d(\beta_{1})(x_{1}~\wedge~u_{2}~\wedge~\cdots~\wedge~u_{n-1},~y_{1}~\wedge~\cdots~\wedge ~y_{n-1},~h_{n})=\\
-\rho(y_{1},~\cdots,~y_{n-1})(\beta_{1}(x_{1},~u_{2},~\cdots,~u_{n-1},~h_{n}))\\
+\beta_{1}(x_{1},~u_{2},~\cdots,~u_{n-1},~\rho(y_{1},~\cdots,~y_{n-1})(h_{n})),$
\item [$\ast$] $d(\beta_{1})(x_{1}~\wedge~u_{2}~\wedge\cdots~\wedge ~u_{n-1},~y_{1}~\wedge~h_{2}~\wedge~\cdots~\wedge ~h_{n-1},~h_{n})=\\
- \displaystyle\sum_{s=2}^{n-1}  \vartheta(y_{1})(h_{2},~\cdots,~h_{s-1},~\beta_{1}(x_{1},~u_{2},~\cdots,~u_{n-1},~h_{s}),~h_{s+1},~\cdots,~h_{n-1},~h_{n})\\
+(-1)^{n-1}\vartheta(x_{1})(u_{2},~\cdots,~u_{n-1},~\beta_{1}(y_{1},~h_{2},~\cdots,~h_{n-1},~h_{n}))\\
-(-1)^{n-1}\vartheta(y_{1})(~h_{2},~\cdots,~h_{n-1},~\beta_{1}(x_{1},~u_{2},~\cdots,~u_{n-1},~h_{n}))\\
-\displaystyle\sum_{s=2}^{n-1}(-1)^{n-1} \beta_{1} (y_{1},~h_{2},~\cdots,~h_{s-1},~\vartheta(x_{1})(u_{2},~\cdots,~u_{n-1},~h_{s}),~h_{s+1},~\cdots,~h_{n-1},~h_{n})\\
+(-1)^{n-1}\beta_{1}(x_{1},~u_{2},~\cdots,~u_{n-1},~\vartheta(y_{1})(h_{2},~\cdots,~h_{n-1},~h_{n}))\\
-(-1)^{n-1}\beta_{1}(y_{1},~h_{2},~\cdots,~h_{n-1},~\vartheta(x_{1})(u_{2},~\cdots,~u_{n-1},~h_{n}))
,$
\item [$\ast$] $d(\beta_{1})(x_{1}~\wedge~u_{2}~\wedge\cdots~\wedge ~u_{n-1},~h_{1}~\wedge~\cdots~\wedge ~h_{n-1},~y_{n})=\\
- \displaystyle\sum_{s=1}^{n-1}  \vartheta(y_{n})(h_{1},~\cdots,~\beta_{1}(x_{1},~u_{2},~\cdots,~u_{n-1},~h_{s}),~\cdots,~h_{n-1})\\
+(-1)^{n-1}\vartheta(x_{1})(u_{2},~\cdots,~u_{n-1},~\beta_{1}(h_{1},~\cdots,~h_{n-1},~y_{n}))\\
- \displaystyle\sum_{s=1}^{n-1}(-1)^{n-1}  \beta_{1} (h_{1},~\cdots,~\vartheta(x_{1})(u_{2},~\cdots,~u_{n-1},~h_{s}),~\cdots,~h_{n-1},~y_{1})\\
+ \beta_{1}(x_{1},~u_{2},~\cdots,~u_{n-1},~\vartheta(y_{n})(h_{1},~\cdots,~h_{n-1}),$
\item [$\ast$] $d(\beta_{1})(u_{1}~\wedge~\cdots~\wedge ~u_{n-1},~y_{1}~\wedge~\cdots~\wedge ~y_{n-1},~z)=\\
-\displaystyle\sum_{s=1}^{n-1} (-1)^{n-s}\rho(y_{1},~\cdots,~y_{s-1},~\hat{y_{s}},~y_{s+1},~\cdots,~y_{n-1},~z)(\beta_{1}(u_{1},~\cdots,~u_{n-1},~y_{s}))
\\- \rho(y_{1},~\cdots,~y_{n-1})(\beta_{1}(u_{1},~\cdots,~u_{n-1},~z))\\+\beta_{1}(u_{1},~\cdots,~u_{n-1},~\pi(y_{1},~\cdots,~y_{n-1},~z)),$
\item [$\ast$]$d(\beta_{1})(u_{1}~\wedge~\cdots~\wedge ~u_{n-1},~y_{1}~\wedge~y_{2}~\wedge~\cdots~\wedge ~y_{n-1},~h_{n})=\\
-(-1)^{n-2}\vartheta(y_{2})(\beta_{1}(u_{1},~\cdots,~u_{n-1},~y_{1}),~h_{3},~\cdots,~h_{n-1},~h_{n})\\
-(-1)^{n-1}\vartheta(y_{1})(\beta_{1}(u_{1},~\cdots,~u_{n-1},~y_{2}),~h_{3},~\cdots,~h_{n-1},~h_{n})\\
 -\beta_{1}(\vartheta(y_{1})(u_{1},~\cdots,~u_{n-1}),~y_{2},~h_{3},~\cdots,~h_{n})\\
-\beta_{1}(y_{1},~\vartheta(y_{2})(u_{1},~\cdots,~u_{n-1}),~h_{3},~\cdots,~h_{n}),$
\item [$\ast$]$d(\beta_{1})(u_{1}~\wedge~\cdots~\wedge ~u_{n-1},~y_{1}~\wedge~h_{2}~\wedge~\cdots~\wedge ~h_{n-1},~y_{n})=\\
-\vartheta(y_{n})(\beta_{1}(u_{1},~\cdots,~u_{n-1},~y_{1}),~h_{2},~\cdots,~h_{n-1})\\
-(-1)^{n-1}\vartheta(y_{1})(h_{2},~\cdots,~h_{n-1},~\beta_{1}(u_{1},~\cdots,~u_{n-1},~y_{n}))\\
 -\beta_{1}(\vartheta(y_{1})(u_{1},~\cdots,~u_{n-1}),~h_{2},~\cdots,~h_{n-1},~y_{n})\\
-\beta_{1}(y_{1},~h_{2},~\cdots,~h_{n-1},~\vartheta(y_{n})(u_{1},~\cdots,~u_{n-1})).$
\end{description}
Hence,~ $d(\beta_{1}+\beta_{2} +\beta_{3})= 0$ if and only if Eqs.\eqref{23}-\eqref{42} hold.
\end{proof}


\begin{thebibliography}{999}
 \bibitem{AI1}  de Azcarraga, J.A. and Izquierdo, J.M., 2009. {\em Cohomology of Filippov algebras and an analogue of Whitehead's lemma.} In Journal of Physics. Conference Series {\bf175}(1) p. 012001

\bibitem {AI2}  de Azcarraga, J.A. and Izquierdo, J.M., 2010. {\em $n$-ary algebras: A review with applications.} Journal of Physics A: Mathematical and Theoretical, {\bf 43}(29), p.293001.

\bibitem{BW}Bai, R., Bai, C. and Wang, J., 2010. {\em Realizations of $3$-Lie algebras.} Journal of mathematical physics, {\bf 51}(6), p.063505.
	
\bibitem {BSZ} Bai, R., Song, G. and Zhang, Y., 2011. {\em On classification of $n$-Lie algebras.} Frontiers of Mathematics in China, {\bf6}(4), pp.581-606.
	
\bibitem{BGS} Bai, C., Guo, L. and Sheng, Y., 2019. {\em Bialgebras, the classical Yang–Baxter equation and Manin triples for $3$-Lie algebras.} Advances in Theoretical and Mathematical Physics,{\bf 23}(1), pp.27-74.
	
\bibitem{CT} Chatterjee, R. and Takhtajan, L., 1996. {\em Aspects of classical and quantum Nambu mechanics.} Letters in Mathematical Physics, {\bf 37}(4), pp.475-482.
	
\bibitem {DT} Daletskii, Y.L. and Takhtajan, L.A., 1997. {\em Leibniz and Lie algebra structures for Nambu algebra.} Letters in Mathematical Physics, {\bf39}(2), pp.127-141.

\bibitem{F} Filippov, V.T., 1985. {\em $n$-Lie algebras.} Siberian Mathematical Journal, {\bf26}(6), pp. 879-891.

\bibitem {FJ} Figueroa-O'Farrill, J.M., 2009. {\em Deformations of $3$-algebras.} Journal of mathematical physics, {\bf 50}(11), p.113514.

\bibitem {G} Gautheron, P., 1996. {\em Some remarks concerning Nambu mechanics.} Letters in Mathematical Physics, {\bf37}(1), pp.103-116.

\bibitem{H}Hoppe, J., 1997. {\em On M-algebras, the quantisation of Nambu-mechanics, and volume preserving diffeomorphisms.} Helvetica Physica Acta, {\bf 70}.
	
\bibitem{LSZB}Liu, J.F., Sheng, Y.H., Zhou, Y.Q. and Bai, C.M., 2016. {\em Nijenhuis Operators on $n$-Lie Algebras.} Communications in Theoretical Physics, {\bf65}(6), p.659.

\bibitem {MA} Makhlouf, A., 2016. {\em On Deformations of $n$-Lie Algebras.} Non-Associative and Non-Commutative Algebra and Operator Theory, p.55.


 \bibitem {NY}  Nambu, Y., 1995. {\em Generalized hamiltonian dynamics.} In Broken Symmetry: Selected Papers of Y Nambu (pp. 302-309).

\bibitem {RM} Rotkiewicz, M., 2005. {\em Cohomology ring of $n$-Lie algebras. Extracta mathematicae}, {\bf 20}(3), pp.219-232.

\bibitem{ST}Song, L. and Tang, R., 2019. {\em Cohomologies, deformations and extensions of $n$-Hom-Lie algebras.} Journal of Geometry and Physics, {\bf141}, pp.65-78.


\bibitem {T}  Takhtajan, L., 1994. {\em On foundation of the generalized Nambu mechanics}. Communications in Mathematical Physics, {\bf160}(2), pp.295-315.

\bibitem {TL}Takhtajan, L.: {\em A higher order analog of Chevalley-Eilenberg complex and deformation theory of $n$-algebras.} St. Petersburg Math. J. {\bf 6}, 429-438 (1995)

  \end{thebibliography}
\end{document}